\documentclass[11pt]{amsart}

\usepackage{amsbsy}
\usepackage{amscd}
\usepackage{amsfonts}
\usepackage{amsgen}
\usepackage{amsmath}
\usepackage{amsopn}
\usepackage{amssymb}
\usepackage{amstext}
\usepackage{amsthm}
\usepackage{amsxtra}
\usepackage[all]{xy}
\usepackage{comment}
\usepackage{euscript}

\usepackage{graphicx}

\usepackage{mathrsfs}
\usepackage{MnSymbol}
\usepackage{rotating}
\usepackage{url}

 \usepackage{hyperref}
\theoremstyle{plain}
\newtheorem{thm}{Theorem}[section]
\newtheorem{prop}[thm]{Proposition}
\newtheorem{lem}[thm]{Lemma}

\theoremstyle{definition}\newtheorem{defn}[thm]{Definition}
\newtheorem{rmk}[thm]{Remark}

\newtheorem{note}[thm]{Notation}

\numberwithin{equation}{section}

\renewcommand{\theta}{\vartheta}
\renewcommand{\phi}{\varphi}
\renewcommand{\epsilon}{\varepsilon}
\renewcommand{\subset}{\subseteq}

\newcommand{\N}{\mathbb N}
\newcommand{\Z}{\mathbb Z}
\newcommand{\Q}{\mathbb Q}
\newcommand{\R}{\mathbb R}
\newcommand{\C}{\mathbb C}

\newcommand{\mm}[1]{\mathrm{#1}}

\newcommand{\mc}[1]{\mathcal{#1}}
\newcommand{\frk}[1]{\mathfrak{#1}}

\newcommand{\ra}{\rightarrow}

\newcommand{\tr}{\mathrm{tr}}
\newcommand{\Tr}{\mathrm{Tr}}

\newcommand{\PM
}{\mathrm{PM}}

\newcommand{\id}{\mm{id}}

\begin{document}

\title[FDE Z-scores of RM in 2D ARMA]{Free Deterministic Equivalent Z-scores of \\   Compound Wishart models: \\A Goodness of fit test of 2D ARMA models}
\author{Tomohiro Hayase}
\address{Graduate School of Mathematics\\University of Tokyo\\Komaba, Tokyo 153-8914, Japan}
\keywords{Free probability, compound Wishart matrices,  second order freeness, fluctuatuon of matrices, free deterministic equivalents,2D ARMA model, 2D convolution }
\email{\href{mailto:}{hayase@ms.u-tokyo.ac.jp}}
\date{\today}

\begin{abstract}
	We introduce a new method to qualify the goodness of fit parameter estimation of compound Wishart models. Our method based on the free deterministic equivalent Z-score, which we introduce in this paper.
	Furthermore, an application to  two dimensional autoregressive moving-average model is provided.

 	Our proposal method is a generalization of statistical hypothesis testing to one dimensional moving average model based on fluctuations of real compound Wishart matrices, which is a recent result by  Hasegawa, Sakuma and Yoshida.   
	
\end{abstract}

\maketitle

\section{Introduction}

Random matrix theory (RMT) has many applications to statistics such as large dimensional models, wireless networks, finance, and quantum information theory (see a review \cite{paul2014random}). 
One of its origin is  the spectral analysis of the sample covariance matrices by Marchenko and Pastur \cite{marvcenko1967distribution}.
For example, the sample covariance matrix of independent sampling can be written as a Wishart random matrix.
A Wishart random matrix \cite{wishart1928generalised} is of the form  $Z^TZ$ where entries of $Z$ are independent and distributed with the normal distribution.

\subsection{Compound Wishart matrices}

In this paper we focus on a class of weighted sum of Wishart matrices which are called compound Wishart matrices.
It  is introduced by Speicher \cite{speicher1998combinatorial}.
The compound Wishart matrices appear as sample covariance matrices of correlated samplings \cite{burda2011applying}.
The compound Wishart matrices and their modifications appear in analysis of some statistical models (\cite{collins2013compound}, \cite{hasegawa2013random} and \cite{couillet2011deterministic}).
Moreover,  the compound Wishart matrices can be seen as  parametrized Wishart matrices. 
Specifically  we define the compound Wishart models as follows:
\begin{defn}[(White real) compound Wishart model]
	Let $(\Omega, \mathcal{F}, P)$ be a probability space.	
	\emph{A compound Wishart model} is a family of random matrices given by 
	\[
	W_{(d/n, D)}:= Z^T D Z,\  (d/n, D) \in \Theta,
	\]
	where $Z$ is $d \times n$ (normalized) random matrix, whose entries  are independently distributed with the normal distribution $\mm{Normal}(0, 1/\sqrt{n})$, and the parameter space  $\Theta$ is  a  subset of the set $\Theta_{\mm{CW}}$ of all parameters defined by $\Theta^{\mm{CW}}  = \bigcup_{n \in \N}\Theta^{\mm{CW}}_n$, where 
	\[
	\Theta^{\mm{CW}}_n:=\bigcup_{d \in \N, d \geq n}  \{ (d/n, D) \in \Q \times M_d(\R) \mid  \text{ $D$ is\ self-adjoint and positve definite} \}.
	\]
\end{defn}
See \cite{hiai2006semicircle} and \cite{redelmeier2014real} for more detail about  the compound Wishart matrices.

\subsection{Fluctuation of random matrices}
Many applications of random matrices rely on computing the asymptotic eigenvalue distribution of large random matrices. 
Recently, more deep result of RMT about the fluctuations of random matrices are  investigated (\cite{collins2007second} and\cite{redelmeier2014real}).
The fundamental fact in the theory is that fluctuation $\Tr(W_{\theta_n}^\ell) - \mathbb{E}[\Tr(W_{\theta_n}^\ell)]  \ (\theta_n \in \Theta^{\mm{CW}}_n)$  is asymptotically normally distributed if the deterministic matrix $D_n$ $(\theta_n =$ $(d_n/n,$ $D_n))$ has the limit moments
$\lim_{n \ra \infty }\mm{tr}(D_n^\ell)  \ (\ell \in \N)$. 
One of the most remarkable facts is the variance of the limit fluctuation
can be written as a polynomial of the limit moments of deterministic parts (see Redelmeier's papers \cite{redelmeier2011genus} and \cite{redelmeier2014real} for real case, and \cite{collins2007second} for the complex case).
Nowadays, this phenomenon is understood as a result of the (real) second order freeness in general situations (see \cite{mingo2006second}\cite{mingo2007second}\cite{collins2007second}\cite{redelmeier2011genus} and \cite{redelmeier2014real}).

Second order freeness contains more information than freeness which is the basic concept in free probability theory (FPT).
FPT is invented by Voiculescu \cite{voiculescu1991limit} which has developed strategies to understand the collective asymptotic behavior of  random matrix ensembles. FPT has provided new results about random matrices as well as different solution of a lot of known results in the random matrix literature.
The strong point of FPT is that freeness separates deterministic matrices and  random ones. 
The reason is that freeness has a role in FPT as independence in the  classical probability theory. Many important random matrix models are asymptotically free; that is, as the size of the matrix becomes large, independent matrices satisfy freeness  on the expected values of the traces of their products \cite{voiculescu1991limit} (see \cite{voiculescu1992free}, \cite{hiai2006semicircle} and \cite{nica2006lectures} for detail).

Several application of spectral analysis of random matrices and freeness have been proposed, but neither fluctuations of random matrices nor second order freeness 
has received as much attention in statistics.
However, recently Hasegawa, Sakuma and Yoshida \cite{hasegawa2013random}\cite{hasegawa2017fluctuations}  apply the fluctuations of compound Wishart matrices in goodness of fit test of one dimensional (1D) autoregressive moving-average (ARMA) models.

\subsection{ARMA model}

The 1D ARMA models are  statistical models of time series, which are studied for a long time. Its hyperparameter selections (in other words, order determinations) is one of the main topics (see Akaike \cite{akaike1974new} and Rissanen \cite{rissanen1978modeling}). The parameter of the models are convolution filters and 
its hyperparameters are size of filters.
Some approaches to 1D ARMA based on RMT and FPT are  provided  (\cite{ZdzisławRandom2010} and \cite{makiej2017spectra}).

Two dimensional (2D) ARMA models are  statistical model used for 2D modeling such 
as \cite{mikhael1994linear} and \cite{zielinski2010two}.
See \cite{aksasse1999rank} and \cite{sadabadi2009two} for  hyperparameter estimation of 2D ARMA.

Goodness of fit test is important to check estimated parameters and hyperparameters in both cases.\\

\subsection{Goodness of fit test by HSY}
The strategy of \cite{hasegawa2017fluctuations} is as follows: instead of using original Z-score  
\[
\mc{Z}_\ell(\theta_n):= ( \Tr(W_{\theta_n}^\ell)-\mathbb{E}[\Tr(W_{\theta_n}^\ell)] )/ \sqrt{\mathbb{V}(\Tr(W_{\theta_n}^\ell))},
\]
consider its limit  $\mc{Z}_\ell := \lim_{n \ra \infty} \mc{Z}_\ell(\theta_n) $ for Z-test of 1D ARMA model 
 if the size $n$ of models is sufficiently large.
One of the biggest benefit of this infinite dimensional approximation is  that  the  variance of the limit fluctuation can be written as a polynomial of the limit moments of deterministic parts (shape parameters) as mentioned above. This nice property makes computation of Z-score quite easier.

Their method works in a good situation such that  the deterministic matrices in models have limit eigenvalue distribution. 
For the 1D ARMA, the shape parameters can be written as  Toeplitz matrices whose limit eigenvalue distributions are determined by Fourier analysis.

\subsection{Free deterministic equivalents}
However, in some cases the limit eigenvalue distribution of deterministic parts possibly does not exist  or is difficult to compute. 
The 2D ARMA models are ones of such models.
To handle this difficulty,  in this paper we introduce an approximation method of goodness of fit test based on \emph{the free deterministic equivalents} (FDE) introduced by Speicher and Vargas \cite{speicher2012free}. Moreover we apply it to 2D ARMA models. 
This method does not require the limit distributions of the deterministic matrices. 

The origin of FDE can be found in Neu and Speciher's paper \cite{neu1995rigorous} as a mean-filed approximation method of an Anderson model in statistical physics.
FDE was rediscovered by \cite{speicher2012free}. The paper pointed out the deterministic equivalent known as an approximation method of Cauchy transform of random matrices in the literature of wireless network \cite{couillet2011deterministic} is a partial realization of FDE.
More precicely, Speicher and Vagas pointed out that considering the approximation of the models at the level of operators based on FPT is more essential than considering it at the level of Cauchy transform.

Despite its rich background in FPT, the algorithm of FDE is not difficult.
It is done by replacing each Gaussian random variable in entries of a random matrix model by an "infinite size" Gaussian random matrix.
Equivalently, FDE is obtained by taking the  limit of the amplified models which is constructed by (1) deterministic matrices are copied by taking tensor product with identity and (2) Gaussian random matrices are  enlarged by  simply increasing the number of i.i.d.\ entries.

\subsection{Our method}
We introduce \emph{the free deterministic equivalent Z-scores} (FDE Z-score) based on FDE as follows.
\begin{defn}
	For each parameter $\theta = (\lambda, D) \in \Theta^{\mm{CW}}$, we define its amplified versions by
	\[
	\theta^\mathfrak{a} = (\lambda, D \otimes I_\mathfrak{a} ), \frk{a} \in \N.
	\]
	Let us define the expected value of FDE and its variance by  
	\[\mu^\Box_\ell(\theta) := \lim_{ \mathfrak{a} \ra \infty }\mathbb{E}[\Tr(W_{\theta^\mathfrak{a}}^\ell)]/\mathfrak{a}, \ \mm{Var}^\Box_\ell(\theta) := \lim_{\mathfrak{a} \ra \infty }\mathbb{V}[\Tr(W_{\theta^\mathfrak{a}}^\ell)],\]
	where  $\Tr$ is the unnormarized trace.
	These limits are well-defined and can be written in polynomials in $\tr(D^k), k \in \N$ (see Lemma \ref{lemma_expected_value} and Lemma \ref{lemma_variance} ). 
	
	\emph{The free deterministic equivalent Z-score} (FDE Z-score, for short) of order $\ell$ for a pair of  a model parameter $\theta$ and a sample parameter $\theta_0$ is  a random variable on $\Omega$ defined by 		
	\[
	\mathcal{Z}_\ell^\Box(\theta_0 \mid \theta):= \frac{\Tr(W_{\theta_0}^\ell)  - \mu_\ell^\Box(\theta)}{\sqrt{\mm{Var}_\ell^\Box(\theta) } }.
	\]
	
\end{defn}

Our main theorem is as follows:
\begin{thm}\label{theorem_main}
	
	Let $\theta_n = (d_n/n, D_n) \in \Theta^{\mm{CW}}_n$$(n \in \N)$ be a sequence of parameters.  Assume that an index $\ell \in \N$ satisfies 
	\begin{align}\label{key_condition}
	R(D_n) := \frac{\|D_n\|}{ \sqrt{\tr(D_n^2)}} =  o(n^{1/{3\ell}}) \mm{ \ as \ } n \ra \infty,
	\end{align}
	where $\|\cdot \|$ is the spectral norm and $\tr$ is the normalized trace.
	Then the law of  $\mathcal{Z}_\ell^\Box(\theta_n \mid \theta_n )$ converges weakly to the standard normal distribution $\mm{Normal}(0,1)$ as $n \ra \infty$.		\qed
\end{thm}

We emphasize that FDE Z-scores do not need to determine the limit eigenvalue distributions of deterministic matrices, and only need  a weaker condition \eqref{key_condition}. 
We introduce an application of this theorem to a goodness of fit test of 2D ARMA models to which the existing method \cite{hasegawa2017fluctuations} cannot be applied.

In addition, our method succeeds befits of \cite{hasegawa2017fluctuations}. At first, it requires less computational costs because we only need to calculate some  moments of parameter matrices. 
Second, it does not depend on estimation methods. At last,
it suggests that the family of moments  can be seen as a usable feature of models.

Before concluding this section, we should note that our method can be generalized to outside of compound Wishart models because many important classes of random matrices have asymptotic normal fluctuations. 

\subsection*{Organization}
In Section 2 we summarize without proofs relevant material on compound Wishart matrices. Section 3 provides a detailed exposition of our main result.
In Section 4 some applications to 2D ARMA models are indicated.
Section 5 presents some numerical simulations of our methods.
Section 6 is devoted to conclusion.

\section{Preliminary}
\subsection{Basic notations}
In this paper we consider a  fixed probability space $(\Omega, \mathcal{F}, P)$.
For a random variable $X$, we denote by $\mathbb{E}[X] = \int X(\omega)d\omega$ the expectation of $X$ and $m_k(X) = \mathbb{E}[X^k]$ the $k$-th moment of $X$.
In addition we denote $\mathbb{V}[X] := m_2(x) - m_1(X)^2$ the variance of $X$.
We denote  by $\Tr$ the trace  and by $\tr = \Tr/N$ the normalized trace of $N \times N$ matrix. We use same symbol $m_k(A) = \tr(A^k)$ for a square matrix $A$ as for a random variable.

\subsection{Cumulants}

	We begin by recalling the basic concepts on partitions and permutations.

\begin{defn}\hfill
	\begin{enumerate}
		\item Set $[n] = \{1,2, \dots, n\}$ for $n \in \N$.
		\item For any finite set $S$, we denote by $|S|$ the number of its elements.
		\item A partition $\pi =\{V_1, \dots, V_k \} \in  2^I$ of a finite set $I$ is a decomposition into mutually disjoint, non-empty  subsets $V_1, \dots, V_k \subset I$.  Those subsets are called blocks of the partition. We denote by $P(I)$ the set of all partitions of $I$.  We write simply $P(n) := P([n])$.
		
	\end{enumerate}

\end{defn}

We introduce the combinatorial form of classical cumulants.

\begin{defn}\label{def_cumulant}
	
	Let $\mathcal{A}$ be the algebra of $\R$-valued random variables who have all moments. 
	Let us define  multi-linear functions $\kappa_\pi : \mathcal{A}^n \ra \R $ $(\pi \in P(n), n \in \N)$ inductively by the following three relations:
	\begin{enumerate}
		\item For $n \in \N$ and $X_1, \dotsc, X_n \in \mathcal{A}$,
		$
		\mathbb{E}[X_1 \dotsb X_n]  = \sum_{\pi \in P(n)} \kappa_\pi [X_1, \dotsc, X_n].
		$
		\item For $n \in \N$ and $\pi \in P(n)$,
		$
		\kappa_{\pi} [X_1, \dotsc, X_n] = %
		\prod_{V \in \pi}\kappa_{(V)} [X_1, \dotsc, X_n].
		$
		\item  For $\pi \in P(n)$ and $V \in \pi$, $\kappa_{(V)} [X_1, \dotsc, X_n] = \kappa_{{\bf1}_m}[X_{j_1}, \dotsc, X_{j_m}]$ where $V= \{j_1 < j_2 < \cdots < j_m\}$.
	\end{enumerate}
	We call them \emph{classical cumulants}. We write $\kappa_n = \kappa_{{\bf1}_n}$ for $n \in \N$. We write $\kappa_n[X] = \kappa_n[X,\dots,X] $($n$-times).
	\qed
\end{defn}

\subsection{Real compound Wishart random matrices}

The partitions and permutations are useful to examine trace of polynomial random matrices.

\begin{defn}\hfill
	\begin{enumerate}
		\item For any subset $J$ of $\N$, we write $-J = \{ -j \in \Z \mid j \in J \}$ and $\pm J = J \cup (-J) \subset \Z \setminus \{ 0\}$. 		
		
		\item We denote by $S(I)$ the permutation group of the finite set $I$.  We write $S_\ell=S([\ell])$.
		For any permutation $\pi$, we denote by $\#\pi$ the number of cycles of $\pi$.
		We use the same symbol $\pi$ for the partition determined by the orbits of a permutation $\pi$.
	\end{enumerate}
	
\end{defn}

\begin{defn}
For  any self-adjoint matrices $A \in M_n(\R)$ and a permutation $\sigma \in S(\pm[n])$ with cycle notation $ \sigma = \gamma_1 \gamma_2 \cdots \gamma_p$, we set
	\[
	\tr_\sigma[A] := \prod_{j =1}^p m_{|\gamma_j|}(A).
	\] 
	We use the same symbols for random matrices.
	\qed
\end{defn}

We recall the notion of premaps. For detail of the relationships of premaps and real Wishart matrices, see Redelmeier's paper~\cite{redelmeier2014real}. 

\begin{defn}\hfill
	\begin{enumerate}
	\item A permutation $\pi \in S(\pm I)$ is said to be a \emph{premap} if $\pi(k) = - \pi^{-1}(-k)$ and no cycle contains both $k$ and $-k$ for any $k \in I$.
	We denote by $\mm{PM}(\pm I)$ the set of all premaps in $S(\pm I)$.
	We write $\mm{PM}_\ell = \mm{PM}(\pm [\ell])$.

\item	For any premap $\pi \in \PM(\pm I)$, a cycle $(c_1, \dots, c_m)$ of $\pi$ is \emph{particular} if $c_{\kappa_0} > 0$ where $\kappa_0 = \mm{argmin}_{k = 1, \dots, m} |c_{k}|$.

\item Denote by $[\pi/2]$ the set of all elements appearing in particular cycles of $\pi$.
We define $\pi/2 \in S([\pi/2])$ by the permutation generated from particular cycles of $\pi$.
		\end{enumerate}
		
\end{defn}

\begin{defn}
	Let $I \subset \Z \setminus \{0 \}$ be a finite set of integers which does not contain both $k$ and $-k$ for any $k \in \N$. 
	\begin{enumerate}
		\item Denote by $\delta \in S(\pm[n])$  the permutation determined by $\delta(-k) = k$.
		\item For $\gamma \in S_\ell$, we define $\gamma_+$, $\gamma_- \in \mm{PM}_\ell$ by $\gamma_+|_{[\ell]} = \gamma$, $\gamma_{-[\ell]} = \id$, $\gamma_-|_{[\ell]} = \id$, $\gamma_-|_{-[\ell]} = -\gamma(-\ast)$.
	\end{enumerate}
\end{defn}

\begin{defn}	
	We define Euler characteristic of $\pi \in \PM_\ell$ with respect to $\gamma \in S_{2\ell}$ by
	\[
	\chi(\gamma, \pi) := \#((\gamma_+ \gamma_-^{-1})/2) + \#(\pi/2) + \#(\gamma_-^{-1}\pi\gamma_+/2) - |I|.
	\]
\end{defn}
%One see that 
%\[
%\chi(\gamma, \delta \pi \delta) =\# \gamma_-^{-1} \pi  \gamma_+/2  - |I|/2  + \# \gamma.
%\]
For the topological meaning of $\chi$, see \cite{redelmeier2014real}.
The following lemma is from \cite{redelmeier2014real}. We use this lemma to explore asymptotic behavior of genus expansion of compound Wishart matrices.

\begin{defn}
		For any $\pi, \gamma \in S(I)$, let us use similar symbols for corresponding partitions. 
		We denote by $\pi \vee \gamma = {\bf1}_{I}$ if for any distinct $k,\ell \in I$, there exist some  $v_1, \dots, v_{2m} \in I$ and $V_1, \dots, V_m \in \pi$ and $W_1, \dots, W_\ell \in \gamma $ with  $ v_{2j}, v_{2j+1} \in V_j$, $v_{2j-1}, v_j \in W_j$ for any $j \leq m$, and $(k,l) = (v_1, v_{2m+1})$ or $(v_{2m+1}, v_1)$.  If $\pi$ satisfies this condition for $\gamma$, we say that \emph{$\pi$ connects all blocks of $\gamma$}.
\end{defn}

\begin{lem}\cite[Lemma~5.2]{redelmeier2014real}\label{lemma_upper_bound_of_eular}
	Let  $\gamma \in S_n$ and $  \{ V_1, \dots, V_r\} \in P(n)$ be the orbits of $\gamma$. Assume that $\pi \in \mm{PM}_n$ connects 	the partition  $\gamma_{\pm}$, that is, $\pi \vee \gamma_\pm = {\bf1}_{\pm[n]}$.
	Then we have  $\chi(\gamma, \pi) \leq 2$.	\qed
\end{lem}

The following lemma directly follows from \cite[Lemma~4.5 and Lemma~3.13]{redelmeier2014real}.
\begin{lem}\label{lemma_genus}
For any $\theta = (d/n, D) \in \Theta$, it holds that
	\[
	\kappa_r[\Tr(W_\theta^\ell)] = \sum_{\substack{\pi \in \mm{PM}_{lr} \\ \pi \vee \gamma_{\pm}  = 1_{\pm [lr]} } }%
	n^{\chi(\gamma,\pi)- r} (\frac{d}{n})^{\#(\pi/2)}\tr_{ \pi/2 }[D].
	\]
	where $\gamma = (1, \dots, \ell)(\ell,\ell+1, \dots, 2\ell) \cdots ((r-1)\ell, \cdots, r\ell)$.
\qed
\end{lem}
\section{Free deterministic equivalent Z-score}

We prove that the FDE Z-score is well-defined.
\begin{lem}\label{lemma_expected_value}
	 It holds that
	\[
	\mu_\ell(\theta) = \lim_{\mathfrak{a} \ra \infty} \mathfrak{a}^{-1} \mathbb{E}[\Tr(W_{\theta^\mathfrak{a}}^l)] = n\alpha_\ell + \beta_\ell,
	\]
	where 
	\begin{align*}
	\alpha_\ell := \sum_{ \substack{ \pi \in \mm{PM}_\ell   \\ \chi( \gamma_\ell,\pi)=2  }}%
	\lambda^{\#(\pi/2)}\tr_{ \pi/2}[D], \
	\beta_\ell := \sum_{ \substack{ \pi \in \mm{PM}_\ell   \\ \chi( \gamma_\ell,\pi)=1  }}%
	\lambda^{\#(\pi/2)}\tr_{ \pi/2}[D], ,
	\end{align*}
	where  $\gamma = (1,2, \dots, \ell)$ and $\lambda = d/n$.
In particular, $\mu_\ell^\Box$ is well-defined and it holds that $\mu_\ell^\Box = nm_1$ and $\mu_2^\Box = n (m_2 +  m_1^2) + m_2$, where $m_k = m_k(D)$ $( k \in \N)$.
	\qed
\end{lem}

\begin{proof}
The first assertions follows from Lemma~\ref{lemma_genus}.
We have $\mm{PM}_2 =  \{ \id, \pi_1 := (1,2)(-1, -2),$ $\pi_2 := (1, -2)(-1,2)\}$.

For $\gamma = (1,2)$, we have
$
\chi(\gamma, \id) = \chi(\gamma, \pi_1 ) = 2, \ \chi(\gamma, \pi_2 ) = 1. 
$
For $\gamma  = \id$, we have
$
\chi(\gamma, \id) = 4,  \ \chi(\gamma, \pi_1 ) = \chi(\gamma, \pi_2 ) = 2. 
$

\end{proof}

\begin{rmk}
	For $n= 1,2$, it holds that 
	$\mathbb{E}[\Tr(W_{\theta}^\ell)] = \mu_\ell^\Box(\theta).$	
\end{rmk}

\begin{lem}\label{lemma_variance}
	It holds that 
\[
\mm{Var}_\ell^\Box(\theta)%
 := \lim_{\mathfrak{a} \ra \infty}\mathbb{V}[\Tr(W_{\theta^\mathfrak{a}}^\ell)] %
 = \sum_{  \substack{\pi \in \mm{PM}_{2\ell} \\ \pi \vee \{\pm V_1, \pm V_2  \}  = 1_{\pm [2\ell]}  \\ \chi(\gamma, \pi) = 2}}  \lambda^{\# \pi/2}\tr_{\pi/2}[D],
\]
where  $\gamma = (1,2, \dots, \ell)(\ell+1, \ell+2, \dots, 2\ell)$ and $V_1 = \{1,2, \dots, l\}$, $V_2=\{\ell+1,\ell+2, \dots, 2\ell\}$ are its blocks, and  $\lambda = d/N$.
In particular, it holds that
$\mm{Var}^\Box_1 = 2 \lambda m_2$ and
$\mm{Var}^\Box_2 = 2(4 \lambda^3 m_1^2m_2 + 2\lambda^2 m_2^2 + 8 \lambda^2 m_1m_3 + 4 \lambda m_4)$, where $m_k = m_k(D) ( k \in \N)$.
\qed
\end{lem}

\begin{proof}
	The first assertion follows from Lemma~\ref{lemma_genus} and that $\tr[(D \otimes I_m)^\ell] = \tr(D^\ell)$.
		
	In the case of $l=1$, $\pi \vee \{\pm V_1, \pm V_2  \}  = 1_{\pm [2\ell]}$ if and only if  
	$\pi/2 = (1,2), (1,-2)$.
	Each $\pi$ satisfies $\chi(\gamma, \pi) = 2$.
	
	In the case of $l=2$, a premap $\pi \in \PM_4$ satisfies $\pi \vee \{\pm V_1, \pm V_2  \}  = 1_{\pm [2\ell]}$ and $\chi(\gamma, \pi) = 2$ if and only if $\pi/2$ is one of the partitions in the following list:
	\begin{align*}
	&(1)(3)(2,4), \ (1)(4)(2,3), \ (2)(3)(1,4), \ (2)(4)(1,3),\\
	&(1,3)(2,4), \ (1,4)(2,3),\\
	&(1)(2,3,4), \ (1)(2,4,3), \ (2)(1,3,4), \ (2)(1,4,3),\\
	&(3)(1,2,4), \ (3)(1,4,2), \ (4)(1,2,3), \ (4)(1,3,2), \\
	& (1,2,3,4), \ (1,2,4,3), \ (1,4,3,2), \ (1,3,4,2),
	\end{align*}
	and 
		\begin{align*}
	&(1)(3)(2,-4), \ (1)(4)(2,-3), \ (2)(3)(1,-4), \ (2)(4)(1,-3),\\
	&(1,-3)(2,-4), \ (1,-4)(2,-3),\\
	&(1)(2,-3,-4), \ (1)(2,-4,-3), \ (2)(1,-3,-4), \ (2)(1,-4,-3),\\
	& (3)(1,2,-4), \ (3)(1,-4,2), \ (4)(1,2,-3), \ (4)(1,-3,2), \\
	& (1,2,-3,-4), \ (1,2,-4,-3), \ (1,-4,-3,2), \ (1,-3,-4,2).
	\end{align*}
By counting the premaps of the same cycle type, we have the computation of the variance.
\end{proof}

\begin{prop}
	$\mm{Var}^\Box_\ell(\theta) \geq \left( dN^{-1}m_2(D) \right)^\ell.$
\end{prop}

\begin{proof}
	Let $\pi \in \mm{PM}_{2\ell}$ such that $\pi/2 = (1,\ell+1)(2,\ell+2) \cdots (\ell,2\ell)$. Then 
$ \pi \vee \{\pm V_1, \pm V_2  \}  = 1_{\pm [2\ell]}$ and $ \chi(\gamma_2, \pi) = 2$. 
Moreover $\tr_{ \pi /2}(D) = m_2(D)^\ell$, which proves the assertion.
\end{proof}

\begin{lem}\label{lemma_main_ineq}
	For any matrix $X$, let us denote by $R(X):=  \|X\|/ \sqrt{\tr(X^*X)}$ the ratio of its operator norm  and its normalized Frobenius norm.
	We denote $\mathcal{Z}_\ell^\Box = \mathcal{Z}_\ell^\Box(\theta | \theta)$ where $\theta = (d/N, D)$. 
	Then the following inequalities hold:
	\begin{align}
		| \kappa_1[\mathcal{Z}_\ell^\Box] | &\leq  R(D)^\ell \frac{ |\mm{ PM}_\ell|}{n}, \label{expectation}\\
	| \kappa_2[\mathcal{Z}_\ell^\Box] - 1 | &\leq R(D)^{2\ell}\frac{ |\mm{PM}_{2\ell}| }{n}, \label{variance}\\
| \kappa_r[\mathcal{Z}_\ell^\Box] | &\leq  R(D)^{r\ell}\frac{ |\mm{ PM}_{r\ell}|}{n^{r-2}} \ (r \geq 3).
\end{align}
\qed
\end{lem}
\begin{proof}
	Let us denote $\gamma_r = (1, 2, \dots, l)(\ell+1, \ell+2, \dots, 2\ell) \dots ((r-2)\ell + 1, (r-1)\ell+2, \dots, r\ell) $
By Lemma~\ref{lemma_genus}, Lemma~\ref{lemma_expected_value} and Lemma~\ref{lemma_variance}, we have
\begin{align}
\kappa_1[\Tr(W_{\theta}^l)]  -  \mu_\ell^\Box(\theta) &= %
\sum_{ \substack{ \pi \in \mm{PM}_\ell   \\ \chi( \gamma_1,\pi) \leq 0 \\ \pi \vee \gamma_1 = 1_{\pm [\ell]}   }}%
	n^{\chi( \gamma_1,\pi) - 1} \left(\frac{d}{n} \right)^{\# \pi/2} \tr_{ \pi/2}[D],\\
\kappa_2[\Tr(W_{\theta}^l)]  - \mm{Var}_\ell^\Box(\theta)  &=%
  \sum_{ \substack{ \pi \in \mm{PM}_{2\ell}   \\ \chi( \gamma_2,\pi) \leq 1 \\ \pi \vee \gamma_2 = 1_{\pm [2\ell]} } }%
 n^{\chi( \gamma_2,\pi) - 2} \left(\frac{d}{n} \right)^{\# \pi/2}\tr_{ \pi /2}[D].
\end{align}
Hence by Lemma~\ref{lemma_upper_bound_of_eular}, $(dn^{-1})^{\# \pi/2}\tr_{ \pi /2}[D] \leq (dn^{-1})^\ell \|D\|^\ell$ and $\mm{Var}^\Box_\ell(\theta) \geq (dn^{-1}m_2(D))^\ell $,  the assertions hold.

\end{proof}

Now we have prepared to prove our main theorem.

\begin{proof}[Proof of Theorem \ref{theorem_main}]
We have $\kappa_r[\mathcal{Z}_\ell^\Box(\theta_n)] = 1 + o(1)$ if $r=2$,
 otherwise $o(1)$, as  $n \ra \infty$ by Lemma~\ref{lemma_main_ineq}.
Hence each cumulant of $\mathcal{Z}_\ell^\Box(\theta_n)$ converges to that of the standard normal distribution, which implies that the law of $\mathcal{Z}_\ell^\Box(\theta_n)$ converges to a standard normal distribution.
\end{proof}

\section{Application to 2D ARMA model}

\begin{defn}
A two dimensional autoregressive moving average (ARMA) model is  a family of random variables $y_{ij}$ with 
\begin{align*}
 \sum_{i=1}^{p_1}\sum_{j=1}^{p_2}a_{ij}y_{h-i + 1, w-j+1} %
 =  \sum_{i=1}^{q_1}\sum_{j=1}^{q_2}b_{ij}\epsilon_{w-i+1, h- j + 1}, \mm{\ for \ any \ } 1\leq h \leq H, 1 \leq w  \leq W,
\end{align*}
where
\begin{enumerate}
\item  orders $p_1, p_2, q_1, q_2 \in \N$, 
\item  AR kernel $a \in M_{p_1, p_2}(\R)$ and MA kernel $b \in M_{q_1, q_2}(\R)$ with $a_{11} =1$ and  $b_{11} \neq  0$,
\item	$\{Z_{ij} \mid i, j \in \Z \}$ is a family of i.i.d.\  random  variables distributed with $\mm{Normal}(0,1)$. 
\end{enumerate}
In the case $p_1 = p_2 = 1$, we call that $\mm{MA}(q_1,q_2)$ model.
\qed
\end{defn}

At first we consider converting  each multi-tuple of MA model  to a compound Wishart model.
Let $y_{i,j}(n)$ be i.i.d.\ $N$ copies of an element distributed with $MA(q_1, q_2)$,  that is, 
\begin{align*}
 y_{h, w}(n) %
 =  \sum_{i=1}^{q_1}\sum_{j=1}^{q_2}b_{ij}\epsilon_{h-i+1, w - j + 1}(n) %
 \ ( 1 \leq h \leq H, \ 1 \leq w \leq W, 1 \leq n \leq N),
\end{align*}
where $\epsilon_{h,w}(n) $ are i.i.d.\ random variables distributed with $\mm{Normal}(0,1)$. 
We write $H_e = H - 1 + q_1, W_e = W + 1 - q_2$.
Define $Y \in M_{HW, N}(\R)$ by $Y_{h (W -1 ) + w, n} = y_{h, w}(n)/\sqrt{N}$.
We define a $HW \times H_eW_e$ matrix $B$ by
\[
B_{ (h-1) W + w , (h + i -2 ) W_e  + j + w -1 }= b_{i, j}.
\]
for $ 1\leq h  \leq H, 1\leq w \leq W$, $1 \leq i \leq q_1$, $1 \leq j \leq q_2$, and the other entries are zero.
Then the law of $Y^TY$ coincides with that of 
\[
 W_{(H_eW_e/N, B^TB)},
\]
as the $M_{N}(\R)$-valued random variables.
At last we define a subset of parameters for two dimensional MA models.

\begin{defn}
	A compound  Wishart model for two dimensional is a family  $(W_\theta)_{\theta \in \Theta_\mm{MA}}$ where  
	\[
	\Theta_\mm{MA} =  \bigcup_{H_e,W_e,N \in \N, H_eW_e \geq N} \{ (H_eW_e/N, B^TB) \mid q_1, q_2 \in \N, b \in M_{q_1, q_2}(\R) , b_{1,1} \neq 0\}.
	\]
	
\end{defn}

\begin{note}
	Let  $c=(c_n)_{n \in \Z}$ be a sequence of real numbers which is zero except for finite number of indexes.
	The Toeplitz random matrix $T_c$ of size $H$ of the sequence is defined by  $T_{i,j}= c_{j-i}$ if $j-i \geq 0$, otherwise $0$ for $i,j \leq H$.
	Let us define Fourier transform of $(c_n)_n$ by 
	\[
	\hat{c}(\xi) = \sum_{n \in Z} c_n \exp (-in\xi), \xi \in [0,2\pi].
	\]
	
	Let $t_c$ be a bounded operator on the Hilbert space $l^2(\Z):=\{(a_n)_{n \in \Z} \mid \|a_n\|_2 < \infty \}$ defined by $(t_c)_{i,j} = c_{j-i}$, where $\|a_n\|_2 := \sqrt{\sum_{ n\in \Z}|a_n|^2 }$. 
	Then $\|T_c\| \leq \|t_c\|  = \|\hat{c}\|_\infty = \sup_{\xi \in [0,2\pi]}|\hat{c}(\xi)| \leq \|c\|_1 = \sum_{n \in \N}|c(n)|$ by the basic results of  Toeplitz operators.
\end{note}

\begin{lem}
	Let $D = B^TB$ and $\zeta= H_eW_e/HW$. Then the following inequalities hold:
	\begin{enumerate}
		\item 	$\tr(D^2) \geq \zeta^{-1}\|b\|_2^2$,
		\item  $\|D\| \leq  \|b\|_1^2$,
		\item 	$ R(D) \leq \sqrt{\zeta }\|b\|_1^2/ \|b\|_2$.
	\end{enumerate}
\end{lem}
\begin{proof}

	At first $\Tr(B^TBB^TB)$ is equal to 
	\[
	\sum_{(i,i_1,i_2,i_3,j,j_1,j_2,j_3) \in \mathcal{S}} b_{i,j}b_{i_1,j_1}b_{i_2,j_2}b_{i_3,j_3},
	\]
	 where $\mc{S} \subset [q_1]^4 \times [q_2]^4$ is defined as follows: $(i,i_1,i_2,i_3,j,j_1,j_2,j_3) \in \mc{S}$  if there are some $ h, h' \in [H],  w, w' \in [W]$ such that
	 \[
	 	h + i = h' + i_3, w + j = w' + j_3,  h + i_1 = h' + i_2, w + j_1 = w' + j_2.	
	 \]
	Summing up the terms whose indexes correspond to the case $h=h'$ and $w=w'$, we get $HW\|b\|_2^2$.  Hence we have proven (1).

	Next let us define $HW \times H_eW_e$ matrix $V$ by 
	\[
	V_{ (h-1) W + 1 , ( h  - 1 ) W_e   + w  }= 1,
	\]
	for $1 \leq w \leq W$ and $1 \leq h \leq H$,
	and the other entries are zero.
	Let us define $H_eW_e \times H_eW_e$  matrix $T$ by
	for $1 \leq w \leq W_e$ and $1 \leq h \leq H_e$,	
	\[
	T_{ (h-1) W_e + w , (h + i -2 ) W_e  + j + w -1 }= b_{i, j}.
	\]
	where $ 1 \leq i \leq q_1 $, $1 \leq j \leq q_2 $ such that $(h + i -2 ) W_e  + j + w - 1 \leq H_eW_e$ and the other entries are zero.
	Then we have $B = VT$.
	Moreover, let us define a sequence $c =(c_n)_{0 \leq n \leq H_eW_e}$ by $c_{(i-1)W_e + j -1 } := b_{i,j}$ for $1 \leq i \leq q_1$ and $1 \leq j \leq q_2$ other wise $0$. Then the matrix $T$ is equal to the upper triangular Toeplitz matrix $T_c$.
	Let us denote $c_i(n) = b_{i,n}$, and $S$ be a $H_eW_e \times H_eW_e$-nilpotent matrix $S$ such that $S_{i,i+1} = 1 ( 1 \leq i \leq H_eW_e)$, and the orher entries are $0$. Then, there is an $m \in \N$ such that
	\[
	T_c =  \sum_{i = 1}^{q_1} S^{(i-1)m}T_{c_i}.
	\]
	Hence $\|T_c\| \leq  \sum_{i = 1}^{q_1} \|T_{c_i}\|$, and
	 $\|D\| \leq \|T_c\|^2\|V^TV\| \leq \|T_c\|^2 \leq (\sum_{i = 1}^{q_1} \|T_{c_i}\|)^2$.
	 Since $\|T_{c_i}\|  \leq  \|c_i\|_1$, the assertion (2) holds.

	The last claim (3) directly follows from (1) and (2).

\end{proof}

\begin{thm}\label{theorem_MA}
	For a fixed MA filter $b \in M_{q_1, q_2}(\R)$, 
	let $\theta_N = (H_eW_e/N, B^TB)$ constructed as above.
	Then the law of $Z_\ell^\Box$ converges weakly  to  $\mm{Normal}(0,1)$ as $N, HW \ra \infty$ with $HW \geq N$. 
\end{thm}
\begin{proof}
	This follows from the estimation $R(D_N) \leq \sqrt{H_eW_e/HW} \|b\|_1/\|b\|_2 \ra \|b\|_1/\|b\|_2$ as $HW , N \ra \infty$ and Theorem~\ref{theorem_main}.
\end{proof}
\begin{rmk}[Convert ARMA to MA]
For a two dimensional ARMA model, define polynomials of commutative variables $z_1, z_2$ by
	\[P(z_1, z_2 ) = \sum_{i=1}^{p_1}\sum_{j=1}^{p_2} a_{ij}z_1^{i-1} z_2^{j-1}, \ 
	Q(z_1, z_2 ) = \sum_{i=1}^{q_1}\sum_{j=1}^{q_2} b_{ij}z_1^{i-1} z_2^{j-1}.
	\]
Consider the following formal power series:
	\[
	\frac{Q(z_1, z_2) }{P(z_1, z_2)} = \sum_{i=1}^{\infty} \sum_{j=1}^{\infty}g_{ij}z_1^{i-1} z_2^{j-1}.
	\]
Equivalently, coefficients $g_{ij}$ are determined by the following recurrent equations: for any $i,j \in \N$, 
\[
	g_{i,j} = b_{i,j} - \sum_{\substack{ 1 \leq k \leq i,\ 1 \leq l \leq  j, \\ (k,l) \neq (1,1)}} g_{i-k+1, j-l+1}a_{k,l},
	\]
	where we set $b_{ij} = 0$ if $i > q_1$ or $j > q_2$, and set
	$a_{kl} = 0$ if  $k  > p_1$ or $l > p_2$.

The AR kernel is said to be reversible if  $P(z_1, z_2)$ has no zero point in unit disc $\{ (z_1, z_2) \in \C^2 \mid |z_1|^2 + |z_2|^2 \leq 1 \}$ when it is regard as an function on $\C^2$.
Because we are only interested in testing optimized stable ARMA model, we may assume  that each ARMA model has  reversible AR kernel  and we cut $g_{ij}$ by sufficiently large max orders $o_1, o_2$. Then we can treat the model as MA($o_1,o_2$).
\qed
\end{rmk}

\section{Numerical Simulations}
Our algorithm consists of the following steps.
At first for i.i.d.\ $N$ sampling $y(m) \in \R^{HW} ( m= 1,2, \dots, N)$ from a fixed ARMA model, we convert it to a matrix $X \in M_{HW, N}(\R)$ by $X_{hw,m} = y(m)_{h,w}$. We call the index $N$ the batch size of the sample.
For $\ell = 1,2$, we calculate 
\[
\mu_\ell := \Tr\left[(X^T X/N)^\ell\right].
\]

For the test of ARMA model, converting to a MA model if necessary, we compute  
the parameter $\theta=(H_eW_e/N, B ) \in \Theta_{\mm{MA}}$.
Let $m_k = \tr(B^k)$ and $\lambda = H_eW_e/N$, then we have estimations of variance and mean as follows:
\begin{align*}
\mathrm{Var}^\Box_1 &= 2 \lambda m_2, \	\mathrm{Var}^\Box_2 = 2(4 \lambda^3m_1^2m_2 + 2 \lambda^2 m_2^2 + 8 \lambda^2 m_1m_3 + 4 \lambda m_4),\\
\mu^\Box_1 &= N\lambda m_1, \ \mu^\Box_2 = N( \lambda m_2 + \lambda^2 m_1^2) + \lambda m_2.
\end{align*}

At last for $\ell=1,2,$ we calculate the testing statistics 
\[
z_\ell :=  \frac{\mu_\ell - \mu^\Box_\ell}{\sqrt{\mathrm{Var}^\Box_\ell}}.
\]
By Theorem \ref{theorem_MA}, $z_1$, $z_2$ are approximately distributed with  $\mm{Normal}(0,1)$ for a sufficiently large $N$.\\
\subsection{Plot of Z-scores}
At first we  plot Z-score $z_2$ for samples generated by a fixed ARMA model.
Consider  following models.
\begin{enumerate}
\item MA($3,3$) with $b_{ij}$ are generated by uniform distribution on $[-1,1]$ and $b_{1,1} = 1$.
\item ARMA($3,3,3,3$) with reversible  AR kernel. We set max orders $o_1 = o_2 = 24$.
\end{enumerate}
We generate $10000$ realization of $16 \times 16$ data with batch size $16$ from each model, and plot $z_2$ in  figure \ref{fig:zero} and figure \ref{fig:one}, respectively.

\begin{figure}[htbp]
	\centering
	\includegraphics[width=0.45\linewidth]{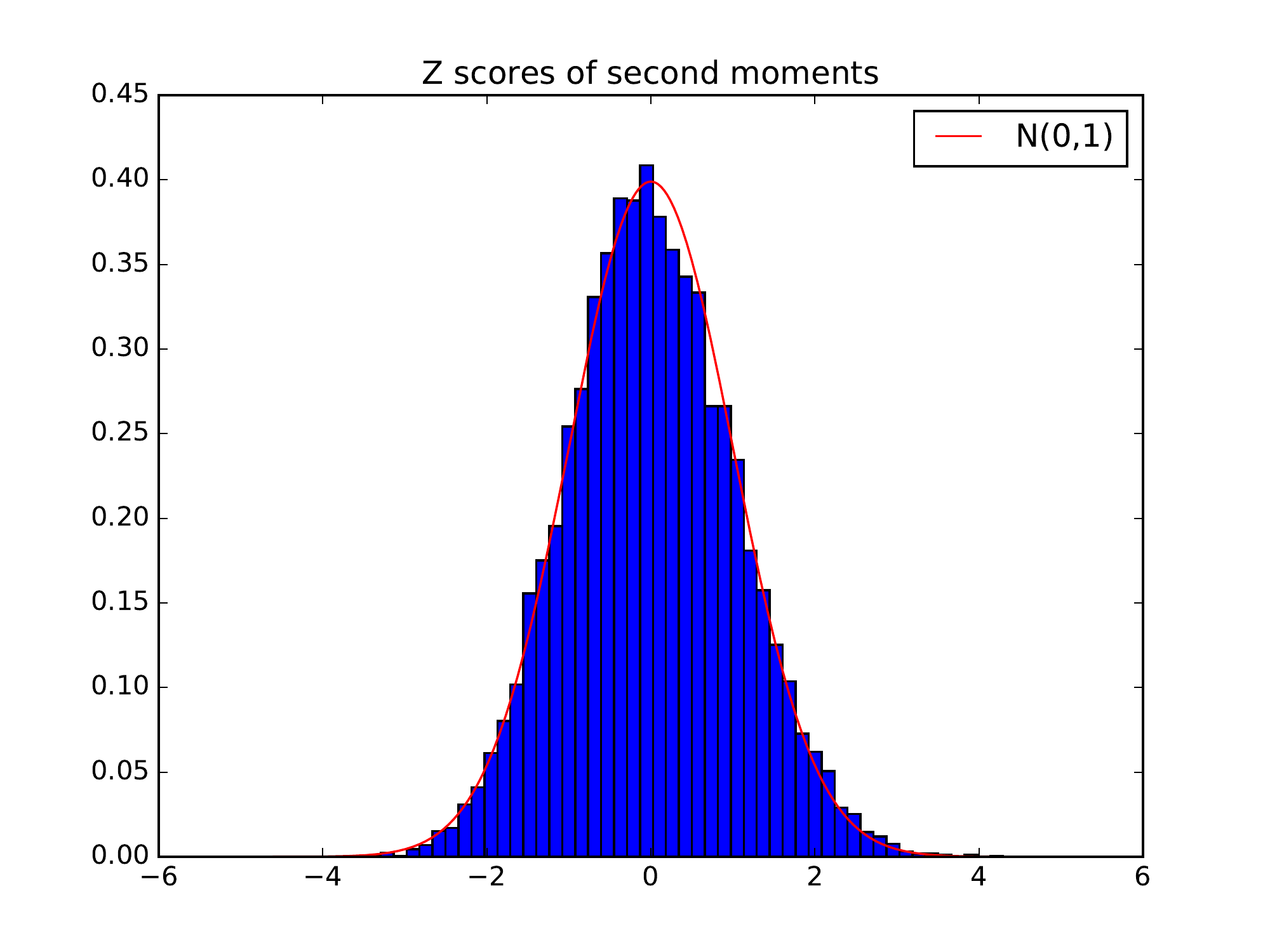}
	\hspace{0.4cm}
	\includegraphics[width=0.45\linewidth]{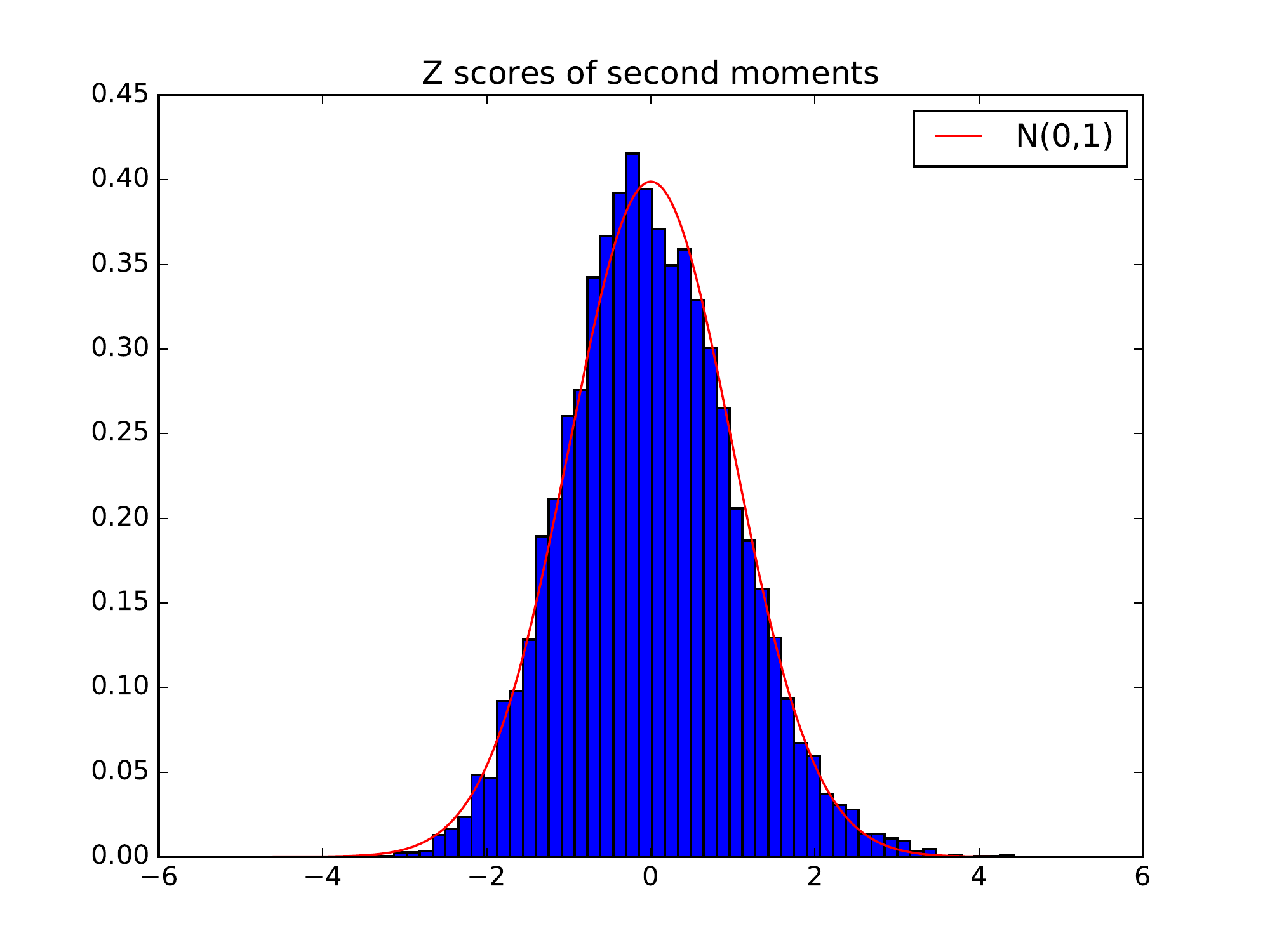}	
 	\caption{Fluctuations of MA(3,3) (left histogram) and of MA(3,3,3,3) (right one)}
	\label{fig:zero}
\end{figure}

\subsection{Goodness of fit test}
Next we observe that our Z-test works. 
We generate $300$ pairs of 2D ARMA models whose AR kernels are trivial or reversible. We assume AR kernel size is less than four, MA kernel size is less than seven, absolute value of each parameter is less than or equal to one, and $\sigma = 1$.
For each pair, we generate a sample with shape $16 \times 16$ and fixed batch size $N$ from one of the ARMA models. We calculate the Z score $z_2$ for each pair. We also calculate $L^1$ distance of their kernels:
\[
d = \sum_{i,j}  | g^{(1)}_{ij} - g^{(2)}_{ij} |,
\]
where we set $g_{ij} = b_{ij}$ for MA models. We set max orders $o_1 = o_2 = 24$ for non MA model. We plot all points $(d + \epsilon, |z_2| + \epsilon)$ from these pairs for $N = 1,16,32,64,128$ and $256$ (we add $\epsilon = 10^{-8}$ to avoid overflow).

\begin{figure}[htbp]
	\centering
	\includegraphics[width=0.45\linewidth]{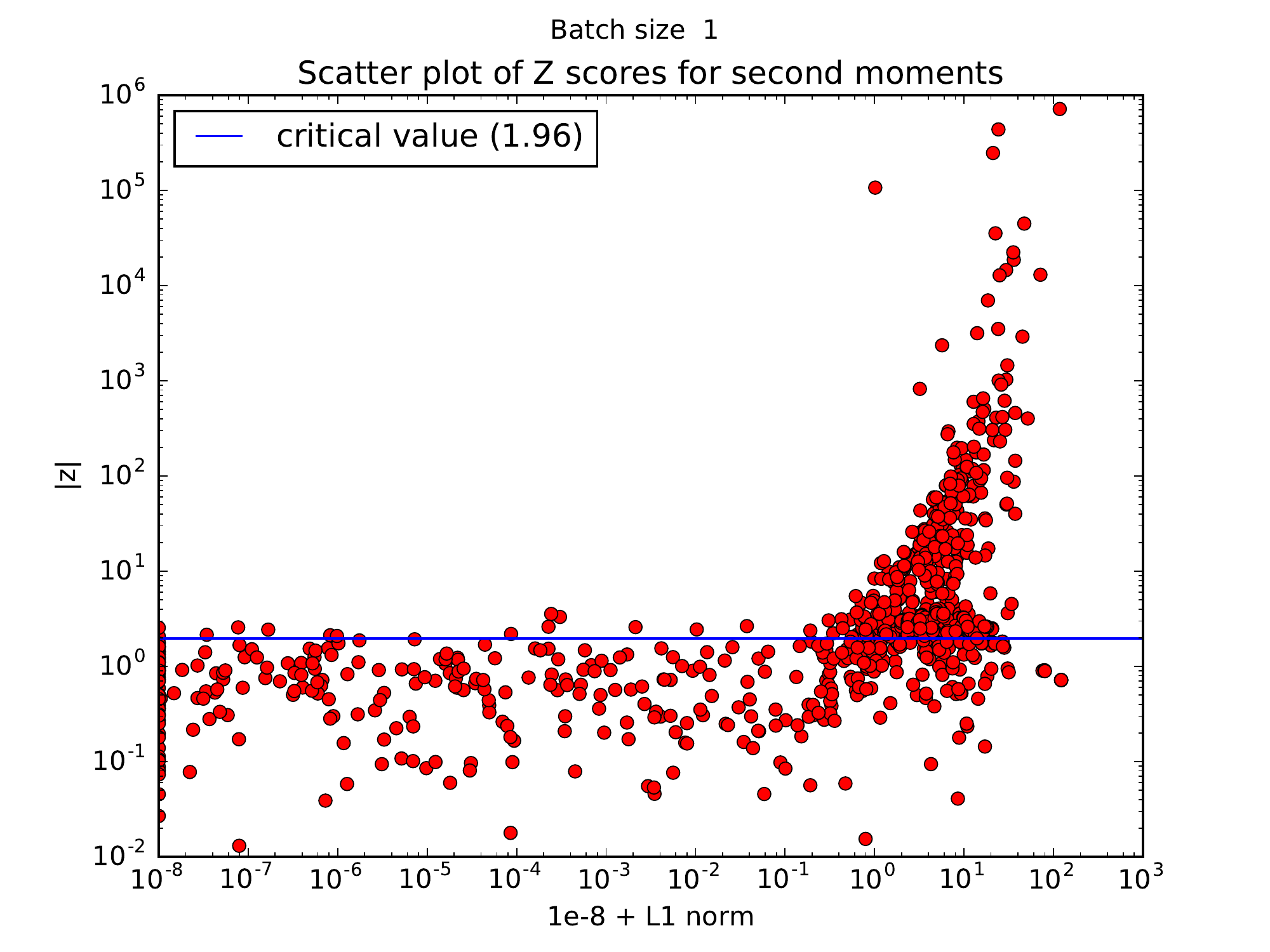}
	\hspace{0.4cm}
	\includegraphics[width=0.45\linewidth]{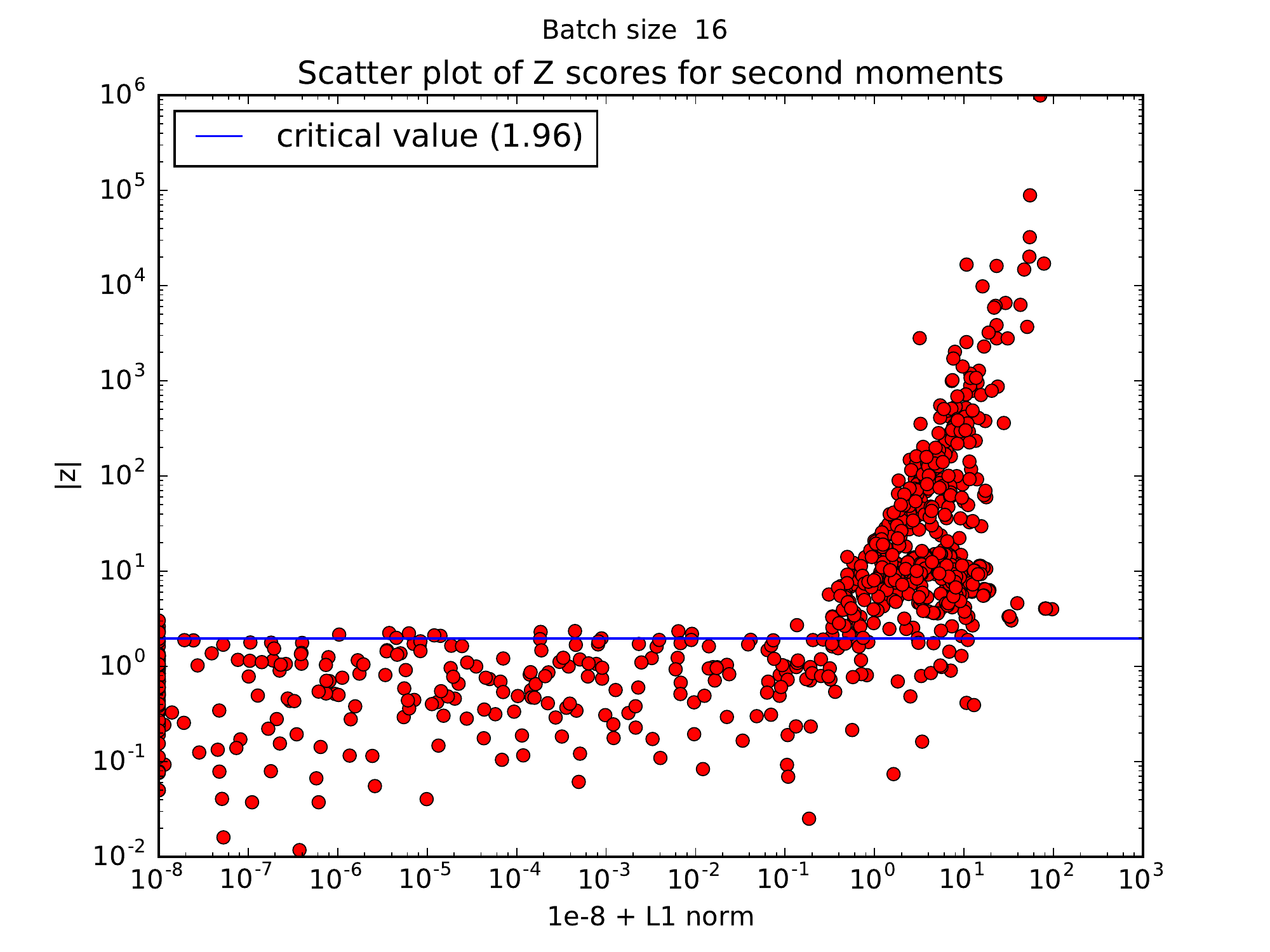}
	\includegraphics[width=0.45\linewidth]{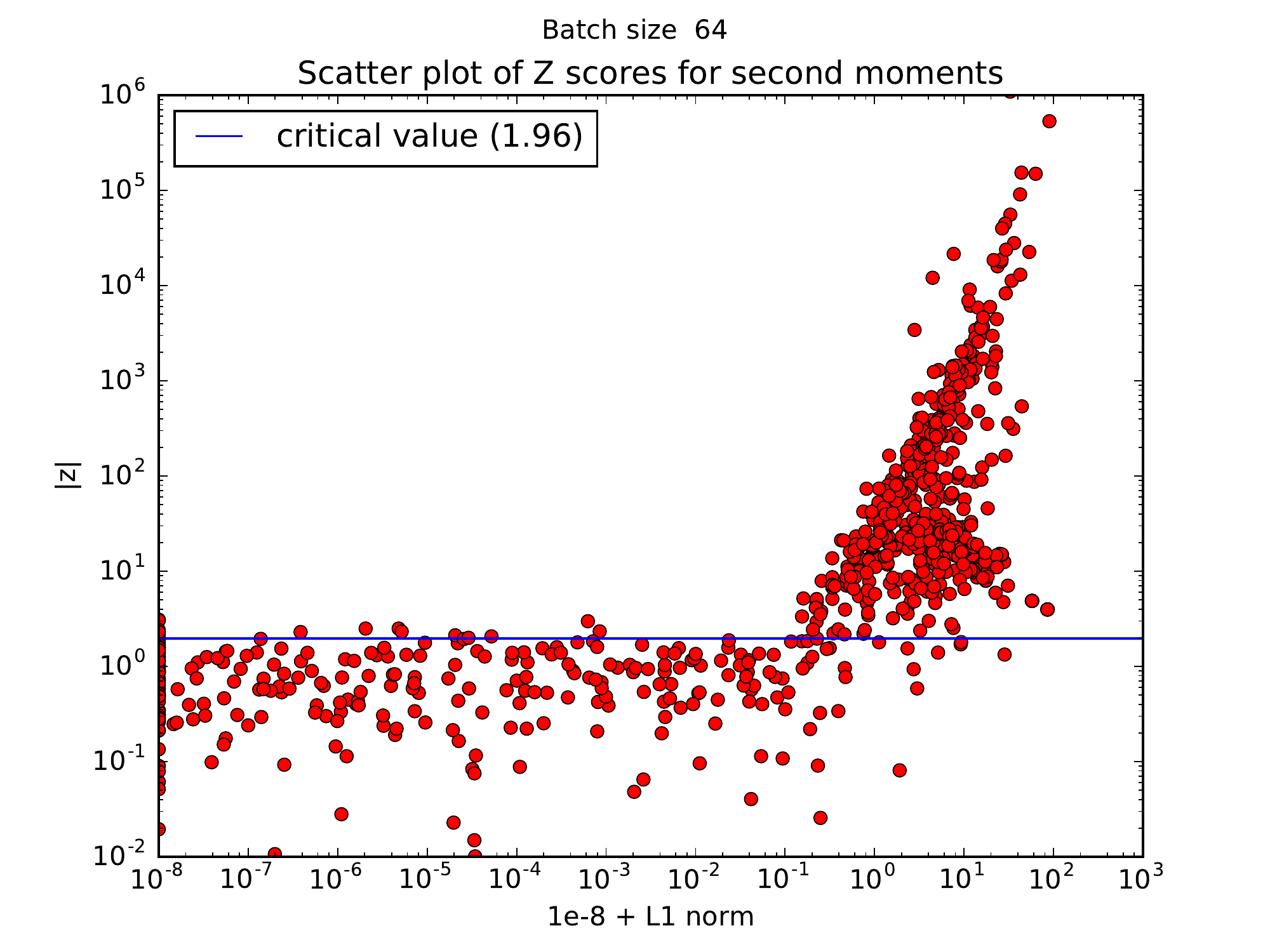}
	\hspace{0.4cm}
	\includegraphics[width=0.45\linewidth]{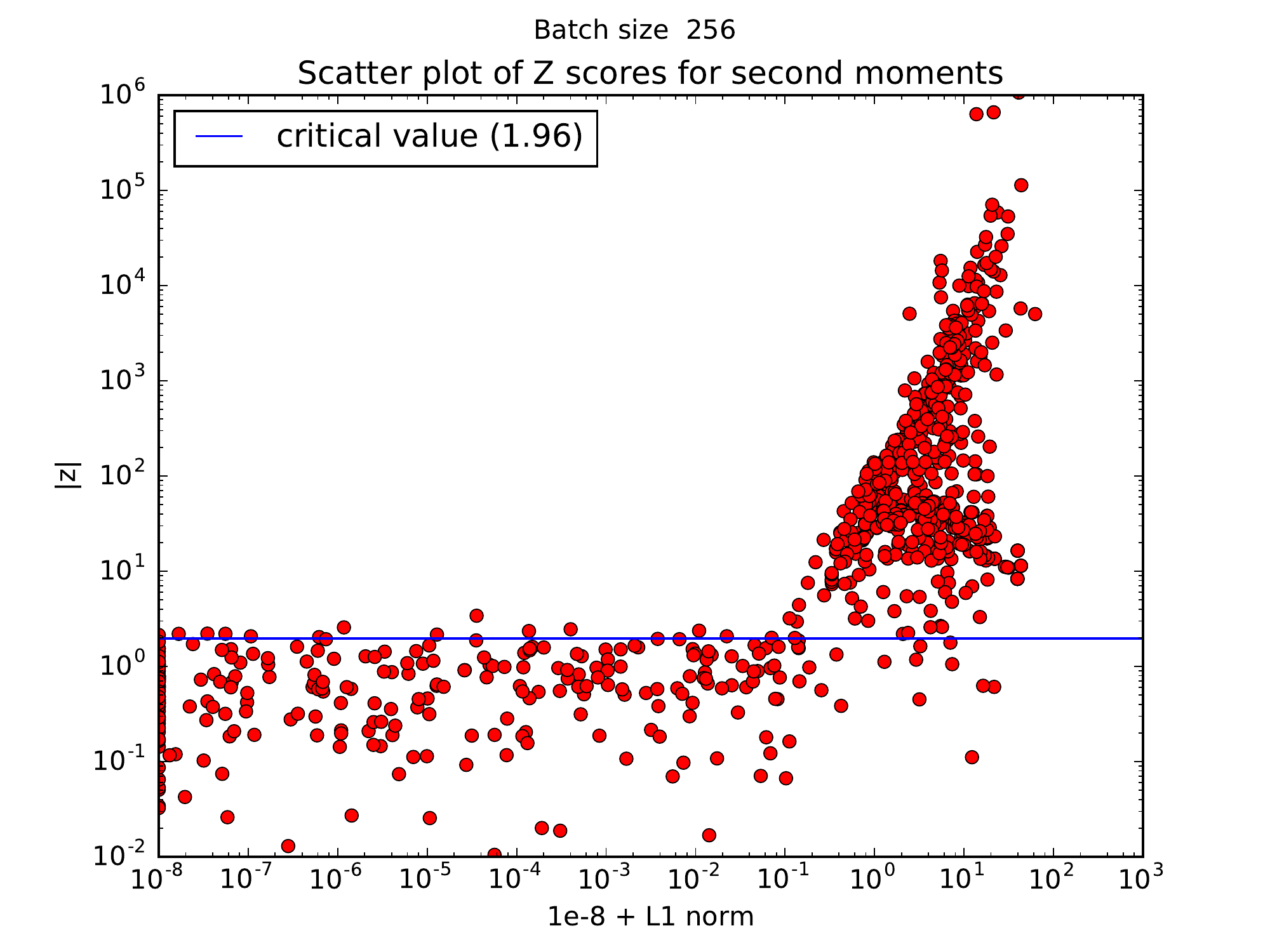}
	\label{fig:1}
	\caption{Z-scores under different batch sizes}
\end{figure}

We show results for $N=1,16,64,256$ in figures \ref{fig:1} , respectively. The horizontal lines in figures are critical value $1.96$
of Z-test.
In our experiment, we define a pair of 2D ARMA model is true (false) if $d < 0.1$ $( d \geq 0.1)$ and the result of hypothesis testing is negative (positive) if $|z_2| > 1.96$ $( |z_2| \leq 1.96)$, respectively. 
We plot in figure \ref{fig:last} the ratios of true-negative/true and false-positive/false  for $10$ realizations of $300$ pairs of ARMA models.
We observe that as the batch size increase, the ratio false-positive/false falls. 
What is noteworthy is that  batch size need not to be large as 256, and batch size from 32 to 64 is enough to archive false-positive/false $\simeq 0.05$. 
The Marchenko-Paster lambda $H_eW_e/N$ for each batch size 32,64 and 256 is  approximately equals to $8, 4$ and $1$, respectively.
Unlike previous study \cite{hasegawa2017fluctuations}, Marchenko-Paster lambda need  not be  close to one and random matrix ensembles need not be close to square for our hypothesis testing.

\begin{figure}[htbp]
	\centering
	\includegraphics[width=0.7\linewidth]{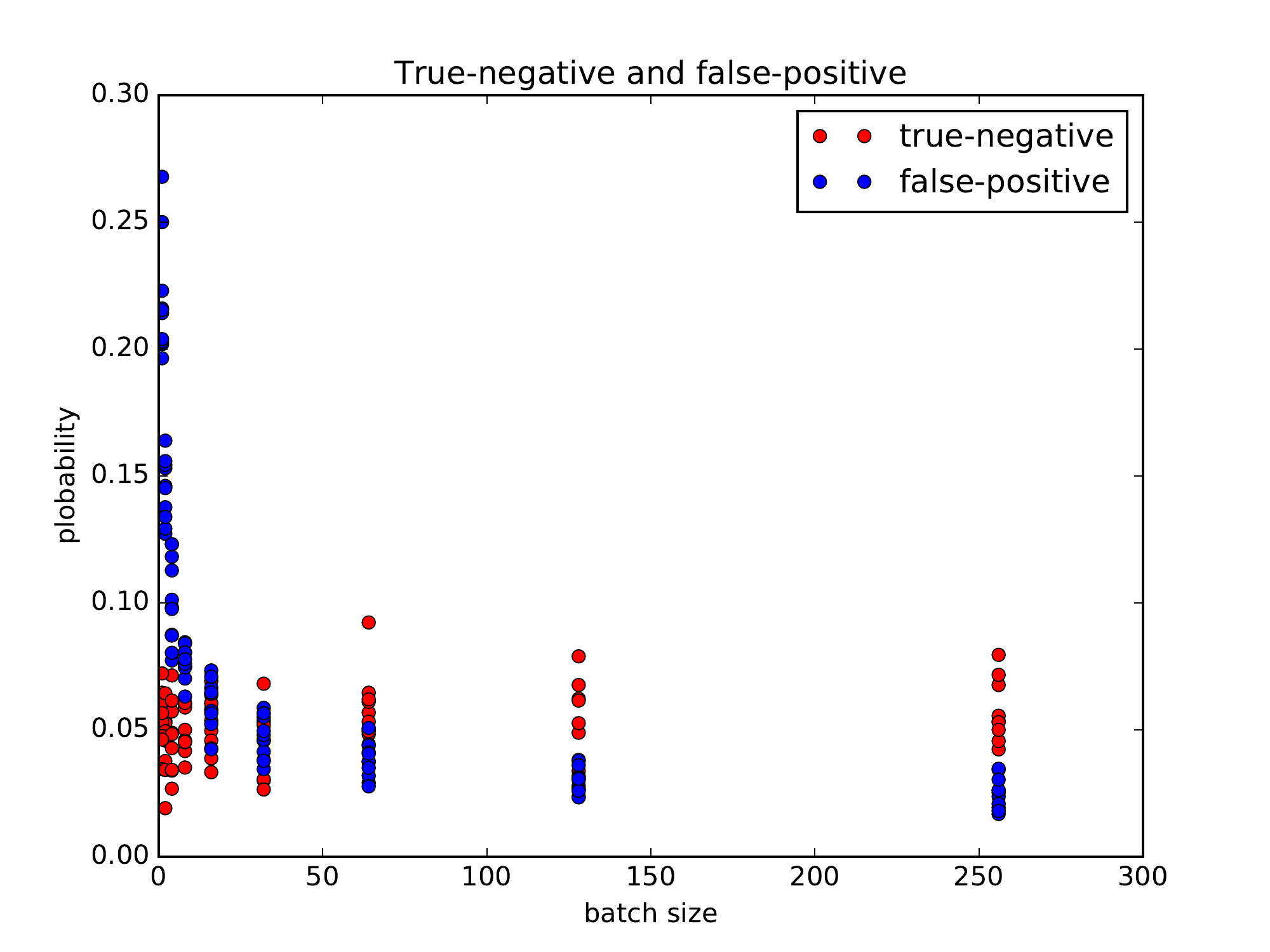}
	\label{fig:1ast}
    \caption{True-negative and false-positive}
\end{figure}

\section{Conclusion}

This paper introduces FDE Z-score,  an approximation of Z-scores of compound Wishart models. 
The key point of our method is the use of FDE to the fluctuations. 
It allows us to efficiently approximate the fluctuations of compound Wishart matrices,
so that we do not need to determine the limit eigenvalue of parameter matrices.
We demonstrate that our method works well for 2D ARMA models.
It turned out that our method does not require too large model size, and works for ARMA model of $16 \times 16$ 2D-data. 
A future direction is to extend its scope  other than compound Wishart models.

\section*{Acknowledgements}
We would like to express our gratitude to Hiroaki Yoshida for incisive comments.
I am grateful to  Noriyoshi Sakuma for carefully proofreading the manuscript.

\bibliographystyle{plain}
\bibliography{reference_2DARMA}

\end{document}